\documentclass{amsart}
\usepackage{amsmath}
\usepackage{amssymb}
\usepackage{amsfonts}

\newtheorem{theorem}{Theorem}
\theoremstyle{plain}

\newtheorem{corollary}{Corollary}

\newtheorem{lemma}{Lemma}

\newtheorem{proposition}{Proposition}
\newtheorem{remark}{Remark}

\numberwithin{equation}{section}
\input{tcilatex}

\begin{document}
\title[Hermite-Hadamard type]{ new inequalities of Hermite-Hadamard type for
functions whose second derivatives absolute values are convex and
quasi-convex}
\author{Mehmet Zeki Sar\i kaya$^{\star }$}
\address{Department of Mathematics,Faculty of Science and Arts, D\"{u}zce
University, D\"{u}zce, Turkey}
\email{ sarikayamz@gmail.com, sarikaya@aku.edu.tr}
\thanks{$^{\star }$corresponding author}
\author{Aziz Saglam}
\address{Department of Mathematics,Faculty of Science and Arts, Afyon
Kocatepe University, Afyon, Turkey}
\email{azizsaglam@aku.edu.tr}
\author{Huseyin Y\i ld\i r\i m}
\address{Department of Mathematics,Faculty of Science and Arts, Afyon
Kocatepe University, Afyon, Turkey}
\email{hyildir@aku.edu.tr}
\date{}

\begin{abstract}
In this paper, we establish several new inequalities for twice
differantiable mappings that are connected with the celebrated
Hermite-Hadamard integral inequality. Some applications for special means of
real numbers are also provided.
\end{abstract}

\maketitle

\section{Introduction}

The following inequality is well known in the literature as the
Hermite-Hadamard integral inequality (see, \cite{PPT}):

\begin{equation}
f\left( \frac{a+b}{2}\right) \leq \frac{1}{b-a}\int_{a}^{b}f(x)dx\leq \frac{%
f(a)+f(b)}{2}  \label{H}
\end{equation}%
where $f:I\subset \mathbb{R}\rightarrow \mathbb{R}$ is a convex function on
the interval $I$ of real numbers and $a,b\in I$ with $a<b$. A function $%
f:[a,b]\subset \mathbb{R}\rightarrow \mathbb{R}$ is said to be convex if \
whenever $x,y\in \lbrack a,b]$ and $t\in \left[ 0,1\right] $, the following
inequality holds%
\begin{equation*}
f(tx+(1-t)y)\leq tf(x)+(1-t)f(y).
\end{equation*}%
This definition has its origins in Jensen's results from \cite{jensen} and
has opened up the most extended, useful and multi-disciplinary domain of
mathematics, namely, canvex analysis. Convex curvers and convex bodies have
appeared in mathematical literature since antiquity and there are many
important resuls related to them. We say that $f$ is concave if $(-f)$ is
convex.

We recall that the notion of quasi-convex functions generalizes the notion
of convex functions. More precisely, a function $f:[a,b]\subset \mathbb{R}%
\rightarrow \mathbb{R}$ is said quasi-convex on $[a,b]$ if 
\begin{equation*}
f(tx+(1-t)y)\leq \sup \left\{ f(x),f(y)\right\}
\end{equation*}%
for all $x,y\in \lbrack a,b]$ and $t\in \left[ 0,1\right] .$ Clearly, any
convex function is a quasi-convex function. Furthermore, there exist
quasi-convex functions which are not convex (see \cite{Ion}).

For several recent results concerning Hermite-Hadamard integral inequality,
we refer the reader to (\cite{Alomari}-\cite{USKMEO}).

In \cite{USK} some inequalities of Hermite-Hadamard type for differentiable
convex mappings were proved using the following lemma.

\begin{lemma}
\label{l1} Let $f:I^{\circ }\subset \mathbb{R}\rightarrow \mathbb{R}$, be a
differentiable mapping on $I^{\circ }$, $a,b\in I^{\circ }$ ($I^{\circ }$ is
the interior of $I$) with $a<b$. If \ $f^{\prime }\in L\left( \left[ a,b%
\right] \right) $, then we have%
\begin{equation}
\begin{array}{l}
\dfrac{1}{b-a}\dint_{a}^{b}f(x)dx-f\left( \dfrac{a+b}{2}\right) \\ 
\\ 
\ \ \ \ \ =\left( b-a\right) \left[ \dint_{0}^{\frac{1}{2}}tf^{\prime
}(ta+(1-t)b)dt+\dint_{\frac{1}{2}}^{1}\left( t-1\right) f^{\prime
}(ta+(1-t)b)dt\right] .%
\end{array}
\label{HH}
\end{equation}
\end{lemma}

One more general result related to (\ref{HH}) was established in \cite%
{USKMEO}. The main result in \cite{USK} is as follows:

\begin{theorem}
\label{l2} Let $f:I\subset \mathbb{R}\rightarrow \mathbb{R}$, be a
differentiable mapping on $I^{\circ }$, $a,b\in I$ with $a<b$. If the
mapping $\left\vert f^{\prime }\right\vert $ is convex on $\left[ a,b\right] 
$, then%
\begin{equation}
\left\vert \frac{1}{b-a}\int_{a}^{b}f(x)dx-f\left( \frac{a+b}{2}\right)
\right\vert \leq \frac{b-a}{4}\left( \frac{\left\vert f^{\prime
}(a)\right\vert +\left\vert f^{\prime }(b)\right\vert }{2}\right) .
\label{H1}
\end{equation}
\end{theorem}

In \cite{CEMPJP}, Pearce and Pe\v{c}ari\'{c} proved the following theorem.

\begin{theorem}
\label{s1} Let $f:I\subset \mathbb{R}\rightarrow \mathbb{R}$, be a
differentiable mapping on $I^{\circ }$, $a,b\in I^{\circ }$ with $a<b$. If
the mapping $\left\vert f^{\prime }\right\vert ^{q}$ is convex on $\left[ a,b%
\right] $ for some $q\geq 1$, then%
\begin{equation}
\left\vert \frac{1}{b-a}\int_{a}^{b}f(x)dx-f\left( \frac{a+b}{2}\right)
\right\vert \leq \frac{b-a}{4}\left( \frac{\left\vert f^{\prime
}(a)\right\vert ^{q}+\left\vert f^{\prime }(b)\right\vert ^{q}}{2}\right) ^{%
\frac{1}{q}}.  \label{H2}
\end{equation}
\end{theorem}

In this article, using functions whose second derivatives absolute values
are convex and quasi-convex, we obtained new inequalities releted to the
left side of Hermite-Hadamard inequality. Finally, we gave some applications
for special means of real numbers.

\section{Hermite-Hadamard type inequalities for convex functions}

We will establish some new results connected with the left-hand side of (\ref%
{H}) used the following Lemma. Now, we give the following new Lemma for our
results:

\begin{lemma}
\label{lm}Let $f:I^{\circ }\subset \mathbb{R}\rightarrow \mathbb{R}$ be
twice differentiable function on $I^{\circ }$, $a,b\in I^{\circ }$ with $%
a<b. $ If $f^{\prime \prime }\in L_{1}[a,b]$, then%
\begin{equation*}
\begin{array}{l}
\dfrac{1}{b-a}\dint_{a}^{b}f(x)dx-f(\dfrac{a+b}{2}) \\ 
\\ 
\ \ \ \ \ \ \ \ \ \ =\dfrac{\left( b-a\right) ^{2}}{2}\dint_{0}^{1}m\left(
t\right) \left[ f^{\prime \prime }(ta+(1-t)b)+f^{\prime \prime }(tb+(1-t)a)%
\right] dt,%
\end{array}%
\end{equation*}%
where%
\begin{equation*}
m(t):=\left\{ 
\begin{array}{ll}
t^{2} & ,t\in \lbrack 0,\frac{1}{2}) \\ 
&  \\ 
\left( 1-t\right) ^{2} & ,t\in \lbrack \frac{1}{2},1].%
\end{array}%
\right.
\end{equation*}
\end{lemma}

\begin{proof}
It suffices to note that%
\begin{equation*}
\begin{array}{lll}
I_{1} & = & \dint_{0}^{1}m\left( t\right) f^{\prime \prime }(ta+(1-t)b)dt \\ 
&  &  \\ 
& = & \dint_{0}^{1/2}t^{2}f^{\prime \prime
}(ta+(1-t)b)dt+\dint_{1/2}^{1}\left( 1-t\right) ^{2}f^{\prime \prime
}(ta+(1-t)b)dt \\ 
&  &  \\ 
& = & \dfrac{1}{a-b}t^{2}f^{\prime }(ta+(1-t)b)\underset{0}{\overset{1/2}{%
\mid }}-\dfrac{2}{a-b}\dint_{0}^{1/2}tf^{\prime }(ta+(1-t)b)dt \\ 
&  &  \\ 
& + & \dfrac{1}{a-b}\left( 1-t\right) ^{2}f^{\prime }(ta+(1-t)b)\underset{1/2%
}{\overset{1}{\mid }}+\dfrac{2}{a-b}\dint_{1/2}^{1}\left( 1-t\right)
f^{\prime }(ta+(1-t)b)dt \\ 
&  &  \\ 
& = & -\dfrac{1}{4\left( b-a\right) }f^{\prime }(\dfrac{a+b}{2})+\dfrac{2}{%
b-a}\left[ \dfrac{1}{a-b}tf(ta+(1-t)b)\underset{0}{\overset{1/2}{\mid }}-%
\dfrac{1}{a-b}\dint_{0}^{1/2}f(ta+(1-t)b)dt\right] \\ 
&  &  \\ 
& + & \dfrac{1}{4\left( b-a\right) }f^{\prime }(\dfrac{a+b}{2})-\dfrac{2}{b-a%
}\left[ \dfrac{1}{a-b}\left( 1-t\right) f(ta+(1-t)b)\underset{1/2}{\overset{1%
}{\mid }}+\dfrac{1}{a-b}\dint_{1/2}^{1}f(ta+(1-t)b)dt\right] \\ 
&  &  \\ 
& = & \dfrac{2}{b-a}\left[ -\dfrac{1}{2\left( b-a\right) }f(\dfrac{a+b}{2})+%
\dfrac{1}{b-a}\dint_{0}^{1/2}f(ta+(1-t)b)dt\right] \\ 
&  &  \\ 
& - & \dfrac{2}{b-a}\left[ \dfrac{1}{2\left( b-a\right) }f(\dfrac{a+b}{2})-%
\dfrac{1}{b-a}\dint_{1/2}^{1}f(ta+(1-t)b)dt\right] \\ 
&  &  \\ 
& = & -\dfrac{2}{\left( b-a\right) ^{2}}f(\dfrac{a+b}{2})+\dfrac{2}{\left(
b-a\right) ^{2}}\dint_{0}^{1}f(ta+(1-t)b)dt.%
\end{array}%
\end{equation*}%
Using the change of the variable $x=ta+(1-t)b$ for $t\in \left[ 0,1\right] ,$
which gives%
\begin{equation}
\begin{array}{lll}
I_{1} & = & -\dfrac{2}{\left( b-a\right) ^{2}}f(\dfrac{a+b}{2})+\dfrac{2}{%
\left( b-a\right) ^{3}}\dint_{a}^{b}f(x)dx.%
\end{array}
\label{d1}
\end{equation}%
Similarly, we can show that%
\begin{equation}
\begin{array}{lll}
I_{2} & = & \dint_{1/2}^{1}m\left( t\right) f^{\prime \prime }(tb+(1-t)a)dt
\\ 
&  &  \\ 
& = & \dint_{0}^{1/2}t^{2}f^{\prime \prime
}(tb+(1-t)a)dt+\dint_{1/2}^{1}\left( 1-t\right) ^{2}f^{\prime \prime
}(tb+(1-t)a)dt \\ 
&  &  \\ 
& = & -\dfrac{2}{\left( b-a\right) ^{2}}f(\dfrac{a+b}{2})+\dfrac{2}{\left(
b-a\right) ^{3}}\dint_{a}^{b}f(x)dx.%
\end{array}
\label{d2}
\end{equation}%
Thus, summing the equalities $\left( \ref{d1}\right) $ and $\left( \ref{d2}%
\right) $, and multiplying the both sides by $\dfrac{\left( b-a\right) ^{2}}{%
4},$ we obtain%
\begin{equation*}
\begin{array}{lll}
\dfrac{\left( b-a\right) ^{2}}{4}\left( I_{1}+I_{2}\right) & = & \dfrac{1}{%
b-a}\dint_{a}^{b}f(x)dx-f(\dfrac{a+b}{2})%
\end{array}%
\end{equation*}%
which is required.
\end{proof}

\begin{theorem}
\label{z1} Let $f:I\subset \mathbb{R}\rightarrow \mathbb{R}$ be twice
differentiable function on $I^{\circ }$ with $f^{\prime \prime }\in
L_{1}[a,b]$. If $\left\vert f^{\prime \prime }\right\vert $ is convex on $%
[a,b],$\ then%
\begin{equation}
\begin{array}{l}
\left\vert \dfrac{1}{b-a}\dint_{a}^{b}f(x)dx-f(\dfrac{a+b}{2})\right\vert
\leq \dfrac{\left( b-a\right) ^{2}}{24}\left[ \dfrac{\left\vert f^{\prime
\prime }\left( a\right) \right\vert +\left\vert f^{\prime \prime }\left(
b\right) \right\vert }{2}\right] .%
\end{array}
\label{d3}
\end{equation}
\end{theorem}

\begin{proof}
From Lemma $\ref{lm}$ and the convexity of $\left\vert f^{\prime \prime
}\right\vert ,$ it follows that 
\begin{equation}
\begin{array}{l}
\left\vert \dfrac{1}{b-a}\dint_{a}^{b}f(x)dx-f(\dfrac{a+b}{2})\right\vert \\ 
\\ 
\text{ \ \ \ \ \ \ \ \ \ }\leq \dfrac{\left( b-a\right) ^{2}}{4}\left\{
\dint_{0}^{1}\left\vert m\left( t\right) \right\vert \left\vert f^{\prime
\prime }(ta+(1-t)b)\right\vert dt+\dint_{0}^{1}\left\vert m\left( t\right)
\right\vert \left\vert f^{\prime \prime }(tb+(1-t)a)\right\vert dt\right\}
\\ 
\\ 
\text{ \ \ \ \ \ \ \ \ \ }\leq \dfrac{\left( b-a\right) ^{2}}{4}\left\{
\dint_{0}^{1}\left\vert m\left( t\right) \right\vert \left[ t\left\vert
f^{\prime \prime }(a)\right\vert +\left( 1-t\right) \left\vert f^{\prime
\prime }(b)\right\vert \right] dt+\dint_{0}^{1}\left\vert m\left( t\right)
\right\vert \left[ t\left\vert f^{\prime \prime }(b)\right\vert +\left(
1-t\right) \left\vert f^{\prime \prime }(b)\right\vert \right] dt\right\} .%
\end{array}
\label{d4}
\end{equation}%
By simple computation, 
\begin{equation}
\begin{array}{l}
\dint_{0}^{1}m\left( t\right) \left[ t\left\vert f^{\prime \prime
}(a)\right\vert +\left( 1-t\right) \left\vert f^{\prime \prime
}(b)\right\vert \right] dt \\ 
\\ 
\text{ \ \ \ \ \ \ \ \ \ }\leq \dint_{0}^{1/2}t^{2}\left[ t\left\vert
f^{\prime \prime }(a)\right\vert +\left( 1-t\right) \left\vert f^{\prime
\prime }(b)\right\vert \right] dt+\dint_{1/2}^{1}\left( 1-t\right) ^{2}\left[
t\left\vert f^{\prime \prime }(a)\right\vert +\left( 1-t\right) \left\vert
f^{\prime \prime }(b)\right\vert \right] dt \\ 
\\ 
\text{ \ \ \ \ \ \ \ \ \ }=\dfrac{\left\vert f^{\prime \prime
}(a)\right\vert +\left\vert f^{\prime \prime }(b)\right\vert }{24}%
\end{array}
\label{d5}
\end{equation}%
and similarly, 
\begin{equation}
\dint_{0}^{1}m\left( t\right) \left[ t\left\vert f^{\prime \prime
}(b)\right\vert +\left( 1-t\right) \left\vert f^{\prime \prime
}(a)\right\vert \right] dt=\dfrac{\left\vert f^{\prime \prime
}(a)\right\vert +\left\vert f^{\prime \prime }(b)\right\vert }{24}.
\label{d6}
\end{equation}%
Using $\left( \ref{d5}\right) $ and $\left( \ref{d6}\right) $ in $\left( \ref%
{d4}\right) $, we obtain $\left( \ref{d3}\right) .$
\end{proof}

\begin{remark}
We note that the obtained midpoint inequality $(\ref{d3})$ is better than
the inequality $(\ref{H1})$.
\end{remark}

Another similar result may be extended in the following theorem.

\begin{theorem}
\label{z3} Let $f:I\subset \mathbb{R}\rightarrow \mathbb{R}$ be twice
differentiable function on $I^{\circ }$ such that $f^{\prime \prime }\in
L_{1}[a,b]$ where $a,b\in I,$ $a<b$. If $\left\vert f^{\prime \prime
}\right\vert ^{q}$ is convex on $[a,b],$\ $q>1$, then%
\begin{equation}
\begin{array}{l}
\left\vert \dfrac{1}{b-a}\dint_{a}^{b}f(x)dx-f(\dfrac{a+b}{2})\right\vert
\leq \dfrac{\left( b-a\right) ^{2}}{8\left( 2p+1\right) ^{1/p}}\left[ \dfrac{%
\left\vert f^{\prime \prime }(a)\right\vert ^{q}+\left\vert f^{\prime \prime
}(b)\right\vert ^{q}}{2}\right] ^{1/q}.%
\end{array}
\label{d07}
\end{equation}
\end{theorem}

\begin{proof}
From Lemma $\ref{lm}$ and using well known H\"{o}lder's integral inequality
, we get, 
\begin{equation*}
\begin{array}{l}
\left\vert \dfrac{1}{b-a}\dint_{a}^{b}f(x)dx-f(\dfrac{a+b}{2})\right\vert \\ 
\\ 
\text{ \ \ \ \ \ \ \ \ \ }\leq \dfrac{\left( b-a\right) ^{2}}{4}\left(
\dint_{0}^{1}\left\vert m\left( t\right) \right\vert ^{p}dt\right)
^{1/p}\left\{ \left( \dint_{0}^{1}\left\vert f^{\prime \prime
}(ta+(1-t)b)\right\vert ^{q}dt\right) ^{1/q}\right. \\ 
\\ 
\text{ \ \ \ \ \ \ \ \ \ }\left. +\left( \dint_{0}^{1}\left\vert f^{\prime
\prime }(tb+(1-t)a)\right\vert ^{q}dt\right) ^{1/q}\right\} .%
\end{array}%
\end{equation*}%
Since $\left\vert f^{\prime \prime }\right\vert ^{q}$ is convex on $[a,b]$,
we known that for $t\in \left[ 0,1\right] $%
\begin{equation*}
\begin{array}{l}
\left\vert f^{\prime \prime }(ta+(1-t)b)\right\vert ^{q}\leq t\left\vert
f^{\prime \prime }(a)\right\vert ^{q}+\left( 1-t\right) \left\vert f^{\prime
\prime }(b)\right\vert ^{q}.%
\end{array}%
\end{equation*}%
Hence,%
\begin{equation*}
\begin{array}{l}
\left\vert \dfrac{1}{b-a}\dint_{a}^{b}f(x)dx-f(\dfrac{a+b}{2})\right\vert \\ 
\\ 
\text{ \ \ \ \ \ \ }\leq \dfrac{\left( b-a\right) ^{2}}{16\left( 2p+1\right)
^{1/p}}\left\{ \left( \dint_{0}^{1}\left[ t\left\vert f^{\prime \prime
}(a)\right\vert ^{q}+\left( 1-t\right) \left\vert f^{\prime \prime
}(b)\right\vert ^{q}\right] dt\right) ^{\frac{1}{q}}+\left( \dint_{0}^{1}%
\left[ t\left\vert f^{\prime \prime }(b)\right\vert ^{q}+\left( 1-t\right)
\left\vert f^{\prime \prime }(a)\right\vert ^{q}\right] dt\right) ^{\frac{1}{%
q}}\right\} \\ 
\\ 
\text{ \ \ \ \ \ \ }=\dfrac{\left( b-a\right) ^{2}}{8\left( 2p+1\right)
^{1/p}}\left[ \dfrac{\left\vert f^{\prime \prime }(a)\right\vert
^{q}+\left\vert f^{\prime \prime }(b)\right\vert ^{q}}{2}\right] ^{1/q},%
\end{array}%
\end{equation*}%
where we have used the fact that 
\begin{equation*}
\dint_{0}^{1}\left\vert m\left( t\right) \right\vert
^{p}dt=\dint_{0}^{1/2}t^{2p}dt+\dint_{1/2}^{1}\left( 1-t\right) ^{2p}dt=%
\frac{1}{4^{p}\left( 2p+1\right) }
\end{equation*}%
which completes the proof.
\end{proof}

An improvement of the constants in Theorem \ref{z3} and a consolidation of
this result with Theorem \ref{z1} are given in the following theorem.

\begin{theorem}
\label{z2} Let $f:I\subset \mathbb{R}\rightarrow \mathbb{R}$ be twice
differentiable function on $I^{\circ }$ such that $f^{\prime \prime }\in
L_{1}[a,b]$ where $a,b\in I,$ $a<b$. If $\left\vert f^{\prime \prime
}\right\vert ^{q}$ is convex on $[a,b],$\ $q\geq 1$, then%
\begin{equation}
\begin{array}{l}
\left\vert \dfrac{1}{b-a}\dint_{a}^{b}f(x)dx-f(\dfrac{a+b}{2})\right\vert
\leq \dfrac{\left( b-a\right) ^{2}}{24}\left[ \dfrac{\left\vert f^{\prime
\prime }\left( a\right) \right\vert ^{q}+\left\vert f^{\prime \prime }\left(
b\right) \right\vert ^{q}}{2}\right] ^{\frac{1}{q}}.%
\end{array}
\label{d7}
\end{equation}
\end{theorem}

\begin{proof}
From Lemma $\ref{lm}$ and using well known power mean inequality , we get, 
\begin{equation*}
\begin{array}{l}
\left\vert \dfrac{1}{b-a}\dint_{a}^{b}f(x)dx-f(\dfrac{a+b}{2})\right\vert \\ 
\\ 
\text{ \ \ \ \ \ \ \ \ \ }\leq \dfrac{\left( b-a\right) ^{2}}{4}\left(
\dint_{0}^{1}\left\vert m\left( t\right) \right\vert dt\right) ^{1/p}\left\{
\left( \dint_{0}^{1}\left\vert m\left( t\right) \right\vert \left\vert
f^{\prime \prime }(ta+(1-t)b)\right\vert ^{q}dt\right) ^{1/q}\right. \\ 
\\ 
\text{ \ \ \ \ \ \ \ \ \ }\left. +\left( \dint_{0}^{1}\left\vert m\left(
t\right) \right\vert \left\vert f^{\prime \prime }(tb+(1-t)a)\right\vert
^{q}dt\right) ^{1/q}\right\} .%
\end{array}%
\end{equation*}%
Since $\left\vert f^{\prime \prime }\right\vert ^{q}$ is convex on $[a,b]$,
we known that for $t\in \left[ 0,1\right] $%
\begin{equation*}
\begin{array}{l}
\left\vert f^{\prime \prime }(ta+(1-t)b)\right\vert ^{q}\leq t\left\vert
f^{\prime \prime }(a)\right\vert ^{q}+\left( 1-t\right) \left\vert f^{\prime
\prime }(b)\right\vert ^{q}.%
\end{array}%
\end{equation*}%
Hence,%
\begin{equation*}
\begin{array}{l}
\left\vert \dfrac{1}{b-a}\dint_{a}^{b}f(x)dx-f(\dfrac{a+b}{2})\right\vert \\ 
\\ 
\text{ \ \ \ \ \ \ \ \ \ }\leq \dfrac{\left( b-a\right) ^{2}}{4}\dfrac{1}{%
\left( 12\right) ^{1/p}}\left\{ \left( \dint_{0}^{1/2}t^{2}\left[
t\left\vert f^{\prime \prime }(a)\right\vert ^{q}+\left( 1-t\right)
\left\vert f^{\prime \prime }(b)\right\vert ^{q}\right] dt\right. \right. \\ 
\\ 
\text{ \ \ \ \ \ \ \ \ \ }\left. +\dint_{1/2}^{1}\left( 1-t\right) ^{2}\left[
t\left\vert f^{\prime \prime }(a)\right\vert ^{q}+\left( 1-t\right)
\left\vert f^{\prime \prime }(b)\right\vert ^{q}\right] dt\right) ^{\frac{1}{%
q}} \\ 
\\ 
\text{ \ \ \ \ \ \ \ \ \ }\left. +\left( \dint_{0}^{1/2}t^{2}\left[
t\left\vert f^{\prime \prime }(b)\right\vert ^{q}+\left( 1-t\right)
\left\vert f^{\prime \prime }(a)\right\vert ^{q}\right] dt+\dint_{1/2}^{1}%
\left( 1-t\right) ^{2}\left[ t\left\vert f^{\prime \prime }(b)\right\vert
^{q}+\left( 1-t\right) \left\vert f^{\prime \prime }(a)\right\vert ^{q}%
\right] dt\right) ^{1/q}\right\} \\ 
\\ 
\text{ \ \ \ \ \ \ \ \ }\ =\dfrac{\left( b-a\right) ^{2}}{4}\dfrac{2}{\left(
12\right) ^{1/p}}\left[ \dfrac{\left\vert f^{\prime \prime }(a)\right\vert
^{q}+\left\vert f^{\prime \prime }(b)\right\vert ^{q}}{24}\right] ^{1/q},%
\end{array}%
\end{equation*}%
where we have used the fact that 
\begin{equation*}
\dint_{0}^{1}\left\vert m\left( t\right) \right\vert
dt=\dint_{0}^{1/2}t^{2}dt+\dint_{1/2}^{1}\left( 1-t\right) ^{2}dt=\frac{1}{12%
}
\end{equation*}%
which completes the proof.
\end{proof}

\begin{remark}
For $q=1,$ this theorem reduces Theorem \ref{z1}. For $q=p/(p-1),\ p>1,\ $we
have an improvement of the constants in Theorem \ref{z3}, since $%
3^{p}>(2p+1) $ if $p>1$ and accordingly%
\begin{equation*}
\frac{1}{24}<\frac{1}{8(2p+1)^{\frac{1}{p}}}.
\end{equation*}
\end{remark}

\begin{remark}
We note that the obtained midpoint inequality $(\ref{d7})$ is better than
the inequality $(\ref{H2}).$
\end{remark}

\section{Hermite-Hadamard type inequalities for quasi-convex functions}

\begin{theorem}
\label{thm} Let $f:I\subset \mathbb{R}\rightarrow \mathbb{R}$ be twice
differentiable function on $I^{\circ }$ such that $f^{\prime \prime }\in
L_{1}[a,b]$ where $a,b\in I,$ $a<b$. If $\left\vert f^{\prime \prime
}\right\vert $ is quasi-convex on $[a,b],$\ then the following inequality
holds:%
\begin{equation}
\begin{array}{l}
\left\vert \dfrac{1}{b-a}\dint_{a}^{b}f(x)dx-f(\dfrac{a+b}{2})\right\vert
\leq \dfrac{\left( b-a\right) ^{2}}{24}\sup \left\{ \left\vert f^{\prime
\prime }\left( a\right) \right\vert ,\left\vert f^{\prime \prime }\left(
b\right) \right\vert \right\} .%
\end{array}
\label{d8}
\end{equation}
\end{theorem}

\begin{proof}
From Lemma $\ref{lm},$ we have 
\begin{equation*}
\begin{array}{l}
\left\vert \dfrac{1}{b-a}\dint_{a}^{b}f(x)dx-f(\dfrac{a+b}{2})\right\vert \\ 
\\ 
\text{ \ \ \ \ \ \ \ \ \ }\leq \dfrac{\left( b-a\right) ^{2}}{4}%
\dint_{0}^{1}\left\vert m\left( t\right) \right\vert \left[ \left\vert
f^{\prime \prime }(ta+(1-t)b)\right\vert +\left\vert f^{\prime \prime
}(tb+(1-t)a)\right\vert \right] dt \\ 
\\ 
\text{ \ \ \ \ \ \ \ \ \ }\leq \dfrac{\left( b-a\right) ^{2}}{4}2\left[
\dint_{0}^{1/2}t^{2}\sup \left\{ \left\vert f^{\prime \prime }(a)\right\vert
,\left\vert f^{\prime \prime }(b)\right\vert \right\}
dt+\dint_{1/2}^{1}\left( 1-t\right) ^{2}\sup \left\{ \left\vert f^{\prime
\prime }(b)\right\vert ,\left\vert f^{\prime \prime }(a)\right\vert \right\}
dt\right] \\ 
\\ 
\text{ \ \ \ \ \ \ \ \ \ }=\dfrac{\left( b-a\right) ^{2}}{24}\sup \left\{
\left\vert f^{\prime \prime }(a)\right\vert ,\left\vert f^{\prime \prime
}(b)\right\vert \right\} .%
\end{array}%
\end{equation*}
\end{proof}

Therefore, we can deduce the following result for quasi-convex functions.

\begin{corollary}
Let $f$ be as in Theorem $\ref{thm}$. Additionally, if
\end{corollary}

$1^{\circ }$ $\left\vert f^{\prime \prime }\right\vert $ is increasing, then
we have 
\begin{equation}
\begin{array}{l}
\left\vert \dfrac{1}{b-a}\dint_{a}^{b}f(x)dx-f(\dfrac{a+b}{2})\right\vert
\leq \dfrac{\left( b-a\right) ^{2}}{24}\left\vert f^{\prime \prime
}(b)\right\vert .%
\end{array}
\label{d9}
\end{equation}

$2^{\circ }$ $\left\vert f^{\prime \prime }\right\vert $ is decreasing, then
we have 
\begin{equation}
\begin{array}{l}
\left\vert \dfrac{1}{b-a}\dint_{a}^{b}f(x)dx-f(\dfrac{a+b}{2})\right\vert
\leq \dfrac{\left( b-a\right) ^{2}}{24}\left\vert f^{\prime \prime
}(a)\right\vert .%
\end{array}
\label{d10}
\end{equation}

\begin{proof}
It follows directly by Theorem $\ref{thm}.$
\end{proof}

\begin{theorem}
\label{thm0} Let $f:I\subset \mathbb{R}\rightarrow \mathbb{R}$ be twice
differentiable function on $I^{\circ }$ such that $f^{\prime \prime }\in
L_{1}[a,b],$ where $a,b\in I,$ $a<b$. If $\left\vert f^{\prime \prime
}\right\vert ^{q}$ is quasi-convex on $[a,b],$\ $q\geq 1$, then the
following inequality holds:%
\begin{equation}
\begin{array}{l}
\left\vert \dfrac{1}{b-a}\dint_{a}^{b}f(x)dx-f(\dfrac{a+b}{2})\right\vert
\leq \dfrac{\left( b-a\right) ^{2}}{8\left( 2p+1\right) ^{\frac{1}{p}}}%
\left( \sup \left\{ \left\vert f^{\prime \prime }(a)\right\vert
^{q},\left\vert f^{\prime \prime }(b)\right\vert ^{q}\right\} \right) ^{1/q}%
\end{array}%
\end{equation}%
where $\frac{1}{p}+\frac{1}{q}=1.$
\end{theorem}

\begin{proof}
From Lemma $\ref{lm}$, using the well known H\"{o}lder's integral inequality
, we have, 
\begin{equation*}
\begin{array}{l}
\left\vert \dfrac{1}{b-a}\dint_{a}^{b}f(x)dx-f(\dfrac{a+b}{2})\right\vert \\ 
\\ 
\text{ \ \ \ \ \ \ \ \ \ }\leq \dfrac{\left( b-a\right) ^{2}}{4}\left(
\dint_{0}^{1}\left\vert m\left( t\right) \right\vert ^{p}dt\right)
^{1/p}\left\{ \left( \dint_{0}^{1}\left\vert f^{\prime \prime
}(ta+(1-t)b)\right\vert ^{q}dt\right) ^{1/q}\right. \\ 
\\ 
\text{ \ \ \ \ \ \ \ \ \ }\left. +\left( \dint_{0}^{1}\left\vert f^{\prime
\prime }(tb+(1-t)a)\right\vert ^{q}dt\right) ^{1/q}\right\} . \\ 
\\ 
\text{ \ \ \ \ \ \ \ \ \ }\leq \dfrac{\left( b-a\right) ^{2}}{16\left(
2p+1\right) ^{\frac{1}{p}}}\left\{ \left( \dint_{0}^{1}\sup \left\{
\left\vert f^{\prime \prime }(a)\right\vert ^{q},\left\vert f^{\prime \prime
}(b)\right\vert ^{q}\right\} dt\right) ^{1/q}+\left( \dint_{0}^{1}\sup
\left\{ \left\vert f^{\prime \prime }(b)\right\vert ^{q},\left\vert
f^{\prime \prime }(a)\right\vert ^{q}\right\} dt\right) ^{1/q}\right\} \\ 
\\ 
\text{ \ \ \ \ \ \ \ \ \ }=\dfrac{\left( b-a\right) ^{2}}{8\left(
2p+1\right) ^{\frac{1}{p}}}\left( \sup \left\{ \left\vert f^{\prime \prime
}(a)\right\vert ^{q},\left\vert f^{\prime \prime }(b)\right\vert
^{q}\right\} \right) ^{1/q}.%
\end{array}%
\end{equation*}%
where we have used the fact that 
\begin{equation*}
\dint_{0}^{1}\left\vert m\left( t\right) \right\vert
^{p}dt=\dint_{0}^{1/2}t^{2p}dt+\dint_{1/2}^{1}\left( 1-t\right) ^{2p}dt=%
\frac{1}{4^{p}\left( 2p+1\right) }
\end{equation*}%
which completes the proof.
\end{proof}

\begin{theorem}
\label{thm1} Let $f:I\subset \mathbb{R}\rightarrow \mathbb{R}$ be twice
differentiable function on $I^{\circ }$ such that $f^{\prime \prime }\in
L_{1}[a,b],$ where $a,b\in I,$ $a<b$. If $\left\vert f^{\prime \prime
}\right\vert ^{q}$ is quasi-convex on $[a,b],$\ $q\geq 1$, then the
following inequality holds:%
\begin{equation}
\begin{array}{l}
\left\vert \dfrac{1}{b-a}\dint_{a}^{b}f(x)dx-f(\dfrac{a+b}{2})\right\vert
\leq \dfrac{\left( b-a\right) ^{2}}{24}\left( \sup \left\{ \left\vert
f^{\prime \prime }(a)\right\vert ^{q},\left\vert f^{\prime \prime
}(b)\right\vert ^{q}\right\} \right) ^{1/q}.%
\end{array}
\label{d11}
\end{equation}
\end{theorem}

\begin{proof}
From Lemma $\ref{lm}$, using the well known power mean inequality , we have, 
\begin{equation*}
\begin{array}{l}
\left\vert \dfrac{1}{b-a}\dint_{a}^{b}f(x)dx-f(\dfrac{a+b}{2})\right\vert \\ 
\\ 
\text{ \ \ \ \ \ \ \ \ \ }\leq \dfrac{\left( b-a\right) ^{2}}{4}\left(
\dint_{0}^{1}\left\vert m\left( t\right) \right\vert dt\right) ^{1/p}\left\{
\left( \dint_{0}^{1}\left\vert m\left( t\right) \right\vert \left\vert
f^{\prime \prime }(ta+(1-t)b)\right\vert ^{q}dt\right) ^{1/q}\right. \\ 
\\ 
\text{ \ \ \ \ \ \ \ \ \ }\left. +\left( \dint_{0}^{1}\left\vert m\left(
t\right) \right\vert \left\vert f^{\prime \prime }(tb+(1-t)a)\right\vert
^{q}dt\right) ^{1/q}\right\} . \\ 
\\ 
\text{ \ \ \ \ \ \ \ \ \ }\leq \dfrac{\left( b-a\right) ^{2}}{4}\dfrac{1}{%
\left( 12\right) ^{1/p}}\left\{ \left( \dint_{0}^{1/2}t^{2}\sup \left\{
\left\vert f^{\prime \prime }(a)\right\vert ^{q},\left\vert f^{\prime \prime
}(b)\right\vert ^{q}\right\} dt\right. \right. \\ 
\\ 
\text{ \ \ \ \ \ \ \ \ \ }\left. +\dint_{1/2}^{1}\left( 1-t\right) ^{2}\sup
\left\{ \left\vert f^{\prime \prime }(a)\right\vert ^{q},\left\vert
f^{\prime \prime }(b)\right\vert ^{q}\right\} dt\right) ^{1/q} \\ 
\\ 
\text{ \ \ \ \ \ \ \ \ \ }\left. +\left( \dint_{0}^{1/2}t^{2}\sup \left\{
\left\vert f^{\prime \prime }(b)\right\vert ^{q},\left\vert f^{\prime \prime
}(a)\right\vert ^{q}\right\} dt+\dint_{1/2}^{1}\left( 1-t\right) \sup
\left\{ \left\vert f^{\prime \prime }(b)\right\vert ^{q},\left\vert
f^{\prime \prime }(a)\right\vert ^{q}\right\} dt\right) ^{1/q}\right\} \\ 
\\ 
\text{ \ \ \ \ \ \ \ \ \ }=\dfrac{\left( b-a\right) ^{2}}{4}\dfrac{1}{\left(
12\right) ^{1/p}}\dfrac{2}{\left( 12\right) ^{1/q}}\left( \sup \left\{
\left\vert f^{\prime \prime }(a)\right\vert ^{q},\left\vert f^{\prime \prime
}(b)\right\vert ^{q}\right\} \right) ^{1/q}.%
\end{array}%
\end{equation*}
\end{proof}

\begin{remark}
For $q=1,$ this theorem reduces Theorem \ref{thm}. For $q=p/(p-1),\ p>1,\ $%
we have an improvement of the constants in Theorem \ref{thm0}, since $%
3^{p}>(2p+1)$ if $p>1$ and accordingly%
\begin{equation*}
\frac{1}{24}<\frac{1}{8(2p+1)^{\frac{1}{p}}}.
\end{equation*}
\end{remark}

\section{Applications to Some Special Means}

We now consider the applications of our Theorems to the following special
means:

(a) The arithmetic mean: $A=A(a,b):=\dfrac{a+b}{2},$ \ $a,b\geq 0,$

(b) The geometric mean: $G=G(a,b):=\sqrt{ab},$ \ $a,b\geq 0,$

(c) The harmonic mean: 
\begin{equation*}
H=H\left( a,b\right) :=\dfrac{2ab}{a+b},\ a,b\geq 0,
\end{equation*}

(d) The logarithmic mean: 
\begin{equation*}
L=L\left( a,b\right) :=\left\{ 
\begin{array}{ccc}
a & if & a=b \\ 
&  &  \\ 
\frac{b-a}{\ln b-\ln a} & if & a\neq b%
\end{array}%
\right. \text{, \ \ \ }a,b>0,
\end{equation*}

(e) The Identric mean:%
\begin{equation*}
I=I\left( a,b\right) :=\left\{ 
\begin{array}{ccc}
a & \text{if} & a=b \\ 
&  &  \\ 
\frac{1}{e}\left( \frac{b^{b}}{a^{a}}\right) ^{\frac{1}{b-a}}\text{ } & 
\text{if} & a\neq b%
\end{array}%
\right. \text{, \ \ \ }a,b>0,
\end{equation*}

(f) The $p-$logarithmic mean

\begin{equation*}
L_{p}=L_{p}(a,b):=\left\{ 
\begin{array}{ccc}
\left[ \frac{b^{p+1}-a^{p+1}}{\left( p+1\right) \left( b-a\right) }\right] ^{%
\frac{1}{p}} & \text{if} & a\neq b \\ 
&  &  \\ 
a & \text{if} & a=b%
\end{array}%
\right. \text{, \ \ \ }p\in \mathbb{R\diagdown }\left\{ -1,0\right\} ;\;a,b>0%
\text{.}
\end{equation*}

It is well known \ that $L_{p}$ is monotonic nondecreasing \ over $p\in 
\mathbb{R}$ with $L_{-1}:=L$ and $L_{0}:=I.$ In particular, we have the
following inequalities%
\begin{equation*}
H\leq G\leq L\leq I\leq A.
\end{equation*}

The following proposition holds:

\begin{proposition}
\label{p.0} Let $a,b\in \mathbb{R}$, $0<a<b$, $n\in \mathbb{Z}$ and $%
\left\vert n(n-1)\right\vert \geq 3$. Then, we have%
\begin{equation*}
\left\vert L_{n}^{n}\left( a,b\right) -A^{n}\left( a,b\right) \right\vert
\leq \left\vert n(n-1)\right\vert \frac{\left( b-a\right) ^{2}}{48}A\left(
a^{(n-2)},b^{(n-2)}\right) .
\end{equation*}
\end{proposition}

\begin{proof}
The proof is immediate from Theorem \ref{z1} applied for $f(x)=x^{n}$, $x\in 
\mathbb{R}$, $n\in \mathbb{Z}$ and $\left\vert n(n-1)\right\vert \geq 3$.
\end{proof}

\begin{proposition}
Let $a,b\in (0,\infty )$ and $a<b$. Then, for all $q>1,$we have%
\begin{equation*}
\ln \left( \frac{I\left( a,b\right) }{A\left( a,b\right) }\right) \leq 
\dfrac{\left( b-a\right) ^{2}}{8a^{2}b^{2}\left( 2p+1\right) ^{\frac{1}{p}}}%
\left[ A\left( a^{2q},b^{2q}\right) \right] ^{\frac{1}{q}}.
\end{equation*}

\begin{proof}
The assertion follows from Theorem \ref{z3} applied to the mapping $%
f:(0,\infty )\rightarrow (-\infty ,0),\ f(x)=-\ln x$ and the details are
omitted.
\end{proof}
\end{proposition}

\begin{proposition}
\label{p.1} Let $a,b\in \mathbb{R}$, $0<a<b$ and $n\in 
%TCIMACRO{\U{2124} }%
%BeginExpansion
\mathbb{Z}
%EndExpansion
$, $\left\vert n(n-1)\right\vert >2$. Then, for all $q>1,$we have%
\begin{equation*}
\left\vert L_{n}^{n}\left( a,b\right) -A^{n}\left( a,b\right) \right\vert
\leq \left\vert n(n-1)\right\vert \frac{\left( b-a\right) ^{2}}{24}\left[
A\left( a^{q(n-2)},b^{q(n-2)}\right) \right] ^{\frac{1}{q}}.
\end{equation*}

\begin{proof}
The assertion follows from Theorem \ref{z2} applied for $f(x)=x^{n}$, $x\in 
\mathbb{R}$, $n\in \mathbb{%
%TCIMACRO{\U{2124} }%
%BeginExpansion
\mathbb{Z}
%EndExpansion
}$ and $\left\vert n(n-1)\right\vert \geq 3$.
\end{proof}
\end{proposition}

\begin{proposition}
\label{p.2} Let $a,b\in \mathbb{R}$, $0<a<b$. Then, for all $q>1$, we have 
\begin{equation*}
\left\vert L^{-1}\left( a,b\right) -A^{-1}\left( a,b\right) \right\vert \leq 
\frac{\left( b-a\right) ^{2}}{24}\frac{2^{\frac{q-1}{q}}}{a^{3}b^{3}}\left[
a^{3q}+b^{3q}\right] ^{\frac{1}{q}}.
\end{equation*}
\end{proposition}

\begin{proof}
The assertion follows from Theorem \ref{z2} applied to $f(x)=\frac{1}{x},\
x\in \lbrack a,b]$ and the details are omitted.
\end{proof}

\begin{proposition}
\label{p.3} Let $a,b\in \mathbb{R}$, $a<b$ and $0\notin \left[ a,b\right] ,$
then, for all $q\geq 1,\ $the following inequality holds:%
\begin{equation*}
\left\vert L^{-1}\left( a,b\right) -A^{-1}\left( a,b\right) \right\vert \leq 
\dfrac{\left( b-a\right) ^{2}}{24}\left( \sup \left\{ \left\vert \frac{2}{%
a^{3}}\right\vert ^{q},\left\vert \frac{2}{b^{3}}\right\vert ^{q}\right\}
\right) ^{\frac{1}{q}}.
\end{equation*}
\end{proposition}

\begin{proof}
The proof is obvious from Theorem \ref{thm1} applied to the quasi-convex
mapping $f(x)=\frac{1}{x},$ $x\in \left[ a,b\right] $.
\end{proof}

\begin{proposition}
\label{p.4} Let $a,b\in \mathbb{R}$, $0<a<b$ and $n\in 
%TCIMACRO{\U{2124} }%
%BeginExpansion
\mathbb{Z}
%EndExpansion
$, $\left\vert n(n-1)\right\vert \geq 3,$ then, for all $q\geq 1,$the
following inequality holds:%
\begin{equation*}
\left\vert L_{n}^{n}\left( a,b\right) -A^{n}\left( a,b\right) \right\vert
\leq \left\vert n(n-1)\right\vert \dfrac{\left( b-a\right) ^{2}}{8\left(
2p+1\right) ^{\frac{1}{p}}}\left( \sup \left\{ a^{q(n-2)},b^{q(n-2)}\right\}
\right) ^{\frac{1}{q}}
\end{equation*}
\end{proposition}

\begin{proof}
The proof is obvious from Theorem \ref{thm0} applied to the quasi-convex
mapping $f(x)=x^{n},$ $x\in \left[ a,b\right] ,\ n\in \mathbb{Z}$ and $%
\left\vert n(n-1)\right\vert \geq 3$.
\end{proof}

\end{document}